\documentclass[12pt]{amsart}
\usepackage{amscd,amssymb}
\usepackage{amsthm,amsmath,enumerate,amssymb}
\usepackage[matrix,arrow]{xy}
\usepackage{enumerate}
\usepackage{longtable}
\usepackage{comment}

\newtheorem*{example*}{Example}
\newtheorem{definition}[equation]{Definition}
\newtheorem{theorem}[equation]{Theorem}

\newtheorem{corollary}[equation]{Corollary}
\newtheorem{proposition}[equation]{Proposition}

\newtheorem*{conjecture*}{Conjecture}
\newtheorem{question}[equation]{Question}
\newtheorem*{question*}{Question}

\newtheorem*{problem*}{Problem}
\newtheorem*{theorem*}{Theorem}
\newtheorem{remark}[equation]{Remark}

\newcommand{\Ric}{{\operatorname{Ric}}}
\newcommand{\vol}{{\operatorname{vol}}}
\newcommand{\idd}{\sqrt{-1}\partial\bar{\partial}}

 \newcommand{\RR}{\mathbb{R}}
\newcommand{\QQ}{\mathbb{Q}}

\makeatletter\@addtoreset{equation}{section} \makeatother

\usepackage{hyperref}
\hypersetup{colorlinks,linkcolor={blue},citecolor={red}}

\title{Continuity of delta invariants and twisted K\"ahler--Einstein metrics}

\begin{document}
\pagestyle{headings}

\author[K.~Zhang]{Kewei Zhang}
\address{Laboratory of Mathematics and Complex Systems, School of Mathematical Sciences, Beijing Normal University, Beijing, 100875, People's Republic of China.}
\email{kwzhang@pku.edu.cn, kwzhang@bnu.edu.cn}

\begin{abstract}
    We show that delta invariant is a continuous function on the big cone. We will also introduce an analytic delta invariant in terms of the optimal exponent in the Moser--Trudinger inequality and prove that it varies continuously in the K\"ahler cone, from which we will deduce the continuity of the greatest Ricci lower bound. Then building on the work Berman--Boucksom--Jonsson, we obtain a uniform Yau--Tian--Donaldson theorem for twisted K\"ahler--Einstein metrics in transcendental cohomology classes.
\end{abstract}

\maketitle
\tableofcontents

\section{Introduction}

\subsection{Background}
Searching for canonical metrics on a given K\"ahler manifold is an important problem in K\"ahler geometry. This paper will largely focus on the twisted K\"ahler--Einstein (tKE) metrics. In the polarized case, it is shown by Boucksom--Jonsson \cite{BJ18} and Berman--Boucksom--Jonsson \cite{BBJ18} that there is an algebraic invariant specifically designed for the existence of tKE metrics. That is the $\delta$-invariant introduced in \cite{FO,BJ17}. However, to study tKE metrics, one does not necessarily need a polarization to begin with. So the major motivation of this paper is to study tKE metrics in a general K\"ahler class. For this purpose, a natural problem to consider would be to extend the definition of $\delta$-invariant to K\"ahler classes and show that such an extension is continuous and characterizes the existence of tKE metrics. This paper aims to give answers to this problem. See \cite{DR17,D17} for related discussions and also \cite{DS,RS} for a completely different approach using Cheeger--Colding theory.

Let $X$ be a compact K\"ahler manifold. 
Let $\xi$ be a K\"ahler class on $X$, and fix a (possibly non-semipositive) smooth form $\alpha\in 2\pi(c_1(X)-\xi)$. Then we would like know whether there exists a K\"ahler form $\omega\in2\pi\xi$ such that the following tKE equation holds:
$$
\Ric(\omega)=\omega+\alpha.
$$
As in the usual Fano case ($\xi=c_1(X)$ and $\alpha=0$), this equation is not always solvable.
When $\xi=c_1(L)$ and $\alpha\geq0$, the existence of such $\omega$ has been successfully characterized in \cite{BBJ18} using the $\delta$-invariant (also referred to as the staility threshold), which we now describe.

Let $X$ be a smooth projective variety. In the literature, $\delta$-invariant is usually defined for $\QQ$-line bundles. But formally, one can extend the definition to $\RR$-line bundles without any trouble. In this paper we shall see that, such an extension is indeed meaningful. Let $L\in N^1(X)_\RR$ be a big $\RR$-line bundle in the N\'eron--Severi space. Following \cite{FO,BJ17}, the $\delta$-invariant $\delta(L)$ of $L$ is defined to be
\begin{equation}
    \delta(L):=\inf_{F}\frac{A(F)}{S_L(F)},
\end{equation}
where $F$ runs through all the prime divisors over $X$, $A(F)$ denotes the log discrepancy of $F$ and $S_L(F)$ denotes the expected vanishing order of $L$ with respect to $F$, i.e.,
\begin{equation}
    S_L(F):=\frac{1}{\vol(L)}\int_0^\infty\vol(L-tF)dt.
\end{equation}
Here $\vol(L-tF)$ makes sense after pulling back $L$ to some birational model containing $F$ (see Section \ref{sec:pre-vol} for more information on the volume function). 

As shown in \cite{BJ18,BBJ18,BLZ}, $\delta$-invariant is the right threshold to detect \emph{Ding-stability}, an algebraic notion designed for the existence of tKE metrics.  More precisely, we can formulate the following  valuative definition for stability. 

\begin{definition}
Let $L$ be an ample $\RR$-line bundle on $X$. $L$ is Ding-semistable (resp. uniformly Ding-stable) if $\delta(L)\geq 1$ (reps. $\delta(L)>1$).
\end{definition}

When $L$ is an ample $\QQ$-line bundle, it is shown in \cite{BJ18} that this valuative definition indeed agrees with the stability notions introduced by Dervan \cite{D16b}. Moreover, the following uniform Yau--Tian--Donaldson theorem for tKE metrics is established by Berman, Boucksom and Jonsson using variational approach.

\begin{theorem}[\cite{BBJ18}]
\label{thm:BBJ-YTD}
Let $L$ be an ample $\QQ$-line bundle on $X$. Fix any smooth form $\alpha\in2\pi(c_1(X)-c_1(L))$ and assume that $\alpha\geq0$. Then we have
\begin{enumerate}
    \item If $\delta(L)>1$, then there exists $\omega\in2\pi c_1(L)$ such that $\Ric(\omega)=\omega+\alpha$.
    
    \item If there exists $\omega\in2\pi c_1(L)$ (resp. a unique $\omega\in2\pi c_1(L)$) such that $\Ric(\omega)=\omega+\alpha$, then $\delta(L)\geq1$ (resp. $\delta(L)>1$).
\end{enumerate}
\end{theorem}

\subsection{Main results}
The main goal of this paper is to show that the above result also holds for transcendental K\"ahler classes.
Note that the assumption $\alpha\geq0$ forces $X$ to be Fano, in which case the K\"ahler cone coincides with the ample cone (since $H^2(X,\RR)\cong H^{1,1}(X,\RR)$). So given any K\"ahler class $\xi$, the $\delta$-invariant $\delta(\xi)$ is well defined (understood as the $\delta$-invariant of the corresponding ample $\RR$-line bundle).
Our main result says the following.

\begin{theorem}
\label{thm:YTD}
Let $\xi$ be a K\"ahler class on $X$. Fix any smooth form $\alpha\in2\pi(c_1(X)-\xi)$ and assume that $\alpha\geq0$. Then we have
\begin{enumerate}
    \item If $\delta(\xi)>1$, then there exists $\omega\in2\pi\xi$ such that $\Ric(\omega)=\omega+\alpha$.
    
    \item If there exists $\omega\in2\pi\xi$ (resp. a unique $\omega\in2\pi\xi$) such that $\Ric(\omega)=\omega+\alpha$, then $\delta(\xi)\geq1$ (resp. $\delta(\xi)>1$).
\end{enumerate}
\end{theorem}

This seems to be the first Yau--Tian--Donaldson type theorem for transcendental cohomology classes. We also remark that the argument in this paper also works for the $\theta$-twisted setting considered in \cite{BBJ18}, where $\theta$ is a semi-positive klt current; see Theorem \ref{thm:theta-YTD}.

Note that, as in \cite{BBJ18}, the positivity assumption on the twist term $\alpha$ guarantees the convexity of twisted Ding (and Mabuchi) functionals.
In general, one would like to drop this assumption and consider tKE equations on possibly non-Fano manifolds. This would be something to pursue in the future. But on Fano manifolds, we do have the following partial answer, which can be easily deduced from the proof of Theorem \ref{thm:YTD}.

\begin{theorem}
\label{thm:YTD-any-alpha}
Let $X$ be a Fano manifold. Let $\xi$ be a K\"ahler class on $X$ and let $s(\xi)=\sup\{s>0|c_1(X)-s\xi>0\}$ be the nef threshold. Then we have
\begin{enumerate}
    \item If $\delta(\xi)\leq s(\xi)$, then for any $\delta\in(0,\delta(\xi))$ and any smooth form $\alpha\in2\pi(c_1(X)-\delta\xi)$, there exists $\omega\in2\pi\xi$ such that $\Ric(\omega)=\delta\omega+\alpha$. 
    
    \item If $\delta(\xi)>s(\xi)$, then for any smooth form $\alpha\in2\pi(c_1(X)-s(\xi)\xi)$, there exists $\omega\in2\pi\xi$ such that
    $\Ric(\omega)=s(\xi)\omega+\alpha$.
\end{enumerate}
\end{theorem}

Now let us briefly explain our strategy.
To prove Theorem \ref{thm:YTD}, we need to establish the continuity of certain stability thresholds so that one can extend the statements for $\QQ$-line bundles to $\RR$-line bundles.
A key input in our argument is the analytic $\delta$-invariant $\delta^A(\cdot)$ to be introduced in Section \ref{sec:delta-A}, which is defined in terms of the optimal exponent in the Moser--Trudinger inequality. As we shall see, both $\delta(\cdot)$ and $\delta^A(\cdot)$ vary continuously on their domains and they are intimately related to the greatest Ricci lower bound \eqref{eq:def-beta-xi}. Moreover we have $\delta^A(L)\leq\delta(L)$ for any ample $\RR$-line bundles (Proposition \ref{prop:delta-A<=delta-R}).
These properties will enable us to conclude the second part (the easier part) of Theorem \ref{thm:YTD}.
Now to prove the first part, we will resort to the argument in \cite[v1]{BBJ18}. More precisely, we need to derive a quantitative lower bound of $\delta^A(L)$ in terms of $\delta(L)$ (see Proposition \ref{prop:delta-A<=delta-R}). Then we can conclude the first part of Theorem \ref{thm:YTD} by the continuity of $\delta$ and $\delta^A$. During the course of this proof, Theorem \ref{thm:YTD-any-alpha} will follow easily.

The next result takes care of the continuity of $\delta$-invariant, which implies that uniform Ding-stablity is an \emph{open condition}.

\begin{theorem}
\label{thm:delta-is-continuous}
The delta invariant $\delta(\cdot)$ is a continuous function on the big cone.
\end{theorem}

Note that, the continuity of Tian's $\alpha$-invariant \cite{T87} has already been shown by Dervan \cite{D15,D16}, whose proof contains two main ingredients. One is the scaling property of $\alpha$, namely, $\alpha(\lambda L)=\lambda^{-1}\alpha(L)$ for any $\lambda>0$. The other is a comparison principle saying that $\alpha(L^\prime)\leq\alpha(L)$ whenever $L^\prime-L$ is effective. To prove the continuity of $\delta$, we will also need these two ingredients. While the scaling property of $\delta$ is clear from the definition, the comparison principle for $\delta$ turns out to somewhat tricky. So instead we will establish a weak comparison principle (see Proposition \ref{prop:delta-comparison}), which is enough for our purpose. Note that the smoothness of $X$ is not required in the proof (it suffices to assume that $X$ is normal projective and has at worst klt singularities). 

\begin{remark}
\rm{
Regarding the continuity of $\delta$, another situation has been considered by Blum--Liu \cite{BL}. They studied a flat family of polarized varieties and showed that $\delta$ is lower semi-continuous in Zariski topology.
}
\end{remark}

We also have the continuity of the analytic $\delta$-invariant (see Section \ref{sec:delta-A} for the definition).
\begin{theorem}\label{thm:delta-A-continuous}
For any compact K\"ahler manifold $X$,
the analytic $\delta$-invariant
    $\delta^A(\cdot)$ is a continuous function on the K\"ahler cone.
\end{theorem}
The proof of this is precisely an analytic version of the argument for Theorem \ref{thm:delta-is-continuous}, which highlights the fact that energy functionals in K\"ahler geometry posses certain non-Archimedean nature (see Proposition \ref{prop:delta-A-comparison}).
Then as a consequence, we obtain the following continuity for the greatest Ricci lower bound (see \eqref{eq:def-beta-xi}), which is not clear at all from its definition.

\begin{theorem}
\label{thm:beta-conti}
For any compact K\"ahler manifold $X$,
the greatest Ricci lower bound $\beta(
\cdot)$ is a continuous function on the K\"ahler cone.
\end{theorem}

This paper is organized as follows.
In Section \ref{sec:pre} we collect some necessary backgrounds for the reader. In Section \ref{sec:delta-A}, we define the analytic $\delta$-invariant and relate it to the greatest Ricci lower bound.
Section \ref{sec:conti} aims to 
establish the continuity of all the stability thresholds appearing in this paper, so in particular, Theorem \ref{thm:delta-is-continuous}, Theorem \ref{thm:delta-A-continuous} and Theorem \ref{thm:beta-conti} are proved. In Section \ref{sec:YTD}, we finish the proof of Theorem \ref{thm:YTD} and then conclude Theorem \ref{thm:YTD-any-alpha}. Finally in Section \ref{sec:delt-delta-A}, we will propose several interesting questions to be considered and we will also briefly discuss the (potential) applications of $\delta$-invariants in the cscK problem. 

\textbf{Acknowledgments.}
The author would like to thank Tam\'as Darvas, Yalong Shi and Feng Wang for many valuable discussions. Thanks also go to Yanir Rubinstein for helpful comments. Special thanks go to Doctor Hattori for pointing out an oversight in \S \ref{sec:cscK} in the previous version.
The author is supported by the China post-doctoral grant BX20190014.

\section{Preliminaries}
\label{sec:pre}

\subsection{The volume function on the N\'eron--Severi space}
\label{sec:pre-vol}
Let $X$ be a normal projective variety of dimension $n$. Its N\'eron--Severi space $N^1(X)_\RR$ consists of numerical equivalence classes of $\RR$-line bundles on $X$, on which one can define a \emph{continuous} volume function $\vol(\cdot)$. For any $L\in N^1(X)_\RR$ and $\lambda>0$, one has
$$
\vol(\lambda L)=\lambda^n\cdot\vol(L).
$$
A line bundle $L\in N^1(X)_\RR$ is called \emph{nef} if for every curve $C$ on $X$,
$
L\cdot C\geq 0.
$
For nef $\RR$-line bundle $L$, $\vol(L)$ is simply equal to the top self-intersection number $L^n$. All the nef elements in $N^1(X)_\RR$ form a convex cone, whose interior is called the \emph{ample cone}.
A class $L\in N^1(X)_\RR$ is called \emph{big} if
$$\vol(L)>0.$$ Note that bigness is an open condition. Namely, given any big $\RR$-line bundle $L$, a sufficiently small perturbation of $L$ in $N^1(X)_\RR$ is big as well. And also, for any big $\RR$-line bundle $B$, one has
\begin{equation}
\label{eq:vol-L+B<vol-L}
    \vol(L+B)\geq\vol(L).
\end{equation}
All the big elements in $ N^1(X)_\RR$ form a convex cone, which is called the big cone.
For more details on this subject, we refer the reader to the standard reference \cite{L04}.

\subsection{The greatest Ricci lower bound}
\label{sec:pre-beta}
Let $\mathcal{K}(X)$ denote the K\"ahler cone of a compact K\"ahler manifold $X$.
For any K\"ahler class $\xi\in\mathcal{K}(X)$, one can define its
greatest Ricci lower bound $\beta(\xi)$ to be
\footnote{We put a factor $2\pi$ for convenience.}
\begin{equation}
    \label{eq:def-beta-xi}
    \beta(\xi):=\sup\{\beta\in\RR\ |\ \exists\text{ K\"ahler form }\omega\in 2\pi\xi\ \text{s.t. }\Ric(\omega)\geq\beta\omega \}.
\end{equation}
When $\xi=c_1(X)$, this invariant was first studied by Tian \cite{T92}.

Observe that $\beta(\xi)$ is bounded from above by the K\"ahler threshold
\begin{equation}
    \label{eq:def-S-xi}
    s(\xi):=\sup\{s\in\RR\ |\ c_1(X)-s\xi\text{ is a  K\"ahler class}\}.
\end{equation}
By definition, $s(\cdot)$ is clearly a continuous function on the K\"ahler cone.
When $\xi=c_1(L)$ for some ample $\RR$-line bundle $L$, we will write
$$
\beta(L):=\beta(c_1(L)),\ s(L):=s(c_1(L))
$$
to ease notation.

When $\xi=2\pi c_1(X)$ it is shown by the author that (see \cite[Appendix]{CRZ})
\begin{equation}
    \beta(-K_X)=\min\{1,\delta(-K_X)\}.
\end{equation}
For general ample $\QQ$-line bundles, we have

\begin{theorem}[\cite{BBJ18}]
\label{thm:BBJ-beta-L}
Let $L$ be an ample $\QQ$-line bundle. Then
$$\beta(L)=\min\{s(L),\delta(L)\}.$$
\end{theorem}

\subsection{Energy functionals and tKE metrics}
\label{sec:pre-energy-func}
Let $(X,\omega)$ be an $n$-dimensional compact K\"ahler manifold. Put
\begin{equation}
    V:=\int_X\omega^n.
\end{equation}
Define
\begin{equation}
    \mathcal{H}(X,\omega):=\{\varphi\in C^\infty(X,\RR)\ |\ \omega_\varphi:=\omega+\idd\varphi>0\}
\end{equation}
and
\begin{equation}
\label{eq:normalized-kahler-potential}
    \mathcal{H}_0(X,\omega):=\{\varphi\in\mathcal{H}(X,\omega)\ |\ \sup_X\varphi=0\}.
\end{equation}
Note that $\Delta\varphi>-n$, so by Green's formula, one can easily find $C_\omega>0$ (only depending on $\omega$) such that
\begin{equation}
    \label{eq:int-vp>-C}
    -C_\omega\leq\int_\omega\varphi\omega^n\leq0,\ \forall\varphi\in\mathcal{H}_0(X,\omega).
\end{equation}
The $I$-functional $I_\omega(\cdot)$ is defined to be
\begin{equation}
    I_\omega(\varphi):=\frac{1}{V}\int_X\varphi(\omega^n-\omega^n_\varphi)=\frac{\sqrt{-1}}{V}\int_X\sum_{i=0}^{n-1}\partial\varphi\wedge\bar{\partial}\varphi\wedge\omega^{n-i-1}\wedge\omega_\varphi^i,\ \varphi\in\mathcal{H}(X,\omega).
\end{equation}
The $J$-functional $J_\omega(\cdot)$ is defined to be
\begin{equation}
\label{eq:J-func}
    J_\omega(\varphi):=\int_0^1\frac{I_\omega(s\varphi)}{s}ds=\frac{\sqrt{-1}}{V}\int_X\sum_{i=0}^{n-1}\frac{n-i}{n+1}\partial\varphi\wedge\bar{\partial}\varphi\wedge\omega^{n-i-1}\wedge\omega_\varphi^i,\ \varphi\in\mathcal{H}(X,\omega).
\end{equation}
We have (see \cite{T87})
\begin{equation}
\label{eq:I-J-lower-upper-J}
    \frac{1}{n}J_\omega\leq I_\omega-J_\omega\leq nJ_\omega.
\end{equation}
Also recall the inequality of Ding \cite{D}:
\begin{equation}
\label{eq:ding-ineq}
    \lambda^{n+1}J_\omega(\varphi)\leq J_\omega(\lambda\varphi)\leq \lambda^{\frac{n+1}{n}}J_\omega(\varphi),\ \forall\lambda\in(0,1),\ \varphi\in\mathcal{H}(X,\omega).
\end{equation}
Moreover, one has the following cocycle relation (see \cite[(14)]{D}):
\begin{equation}
    \label{eq:J-cocycle}
    \begin{split}
        J_\omega(\varphi)-J_\omega(\phi)&=J_{\omega_\phi}(\varphi-\phi)+\frac{1}{V}\int_X(\varphi-\phi)(\omega^n-\omega^n
    _\phi)\\
        &\geq\frac{1}{V}\int_X(\varphi-\phi)(\omega^n-\omega^n
    _\phi),\ \forall\varphi,\phi\in\mathcal{H}(X,\omega).
    \end{split}
\end{equation}

Now fix any smooth form
\begin{equation}
    \alpha\in2\pi c_1(X)-[\omega].
\end{equation}
Then by $\partial\bar{\partial}$-lemma, there exists a unique normalized Ricci potential $f_\alpha\in C^\infty(X,\RR)$ such that
\begin{equation}
\label{eq:def-f_alpha}
    \Ric(\omega)=\omega+\alpha+\idd f_\alpha
    \text{ and }
    \int_Xe^{f_\alpha}\omega^n=V.
\end{equation}
The $\alpha$-twisted Ding functional $D_\alpha$ is defined by
\begin{equation}
    D_\alpha(\varphi):=J_\omega(\varphi)-\frac{1}{V}\int_X\varphi\omega^n-\log\bigg(\frac{1}{V}\int_Xe^{f_\alpha-\varphi}\omega^n\bigg),\ \varphi\in\mathcal{H}(X,\omega).
\end{equation}
And the $\alpha$-twisted Mabuchi functional $M_\alpha$ is defined by
\begin{equation}
    M_\alpha(\varphi):=\frac{1}{V}\int_X\log\frac{\omega_\varphi^n}{\omega^n}\omega_\varphi^n-\frac{1}{V}\int_Xf_\alpha\omega_\varphi^n-(I_\omega-J_\omega)(\varphi),\ \varphi\in\mathcal{H}(X,\omega).
\end{equation}

\begin{definition}
The Ding functional (resp. Mabuchi functional) is called coercive if there exist $\varepsilon,C>0$ such that $D_\alpha\geq \varepsilon J_\omega-C$ (resp. $M_\alpha\geq \varepsilon J_\omega-C$) on $\mathcal{H}(X,\omega)$.
\end{definition}

Regarding the energy functionals and the existence of (twisted) K\"ahler--Einstein metrics, there is now a large literature; see e.g., \cite{T97,TZ,PSSW,R08b,Sze11,B13,DR,BBEGZ}. In this article, we will make use of the following well-established result.

\begin{theorem}
\label{thm:D-M-tKE}
We have the following properties.
\begin{enumerate}
    \item $M_\alpha(\varphi)\geq D_\alpha(\varphi)$ for all $\varphi\in\mathcal{H}(X,\omega)$ and the equality holds iff $\Ric(\omega_\varphi)=\omega_\varphi+\alpha$.
    
    \item $D_\alpha$ is bounded below iff $M_\alpha$ is.
    
    \item $D_\alpha$ is coercive iff $M_\alpha$ is, in which case, there exists $\varphi\in\mathcal{H}(X,\omega)$ such that $\Ric(\omega_\varphi)=\omega_\varphi+\alpha$.
    
    \item Assume that $\alpha\geq0$. If there exists some $\varphi\in\mathcal{H}_0(X,\omega)$ such that $\Ric(\omega_\varphi)=\omega_\varphi+\alpha$, then $D_\alpha/M_\alpha$ is bounded below.
    
    \item Assume that $\alpha\geq0$. If there exists a unique $\varphi\in\mathcal{H}_0(X,\omega)$ such that $\Ric(\omega_\varphi)=\omega_\varphi+\alpha$, then $D_\alpha/M_\alpha$ is coercive.

\end{enumerate}

\end{theorem}

In the last two items, the assumption $\alpha\geq0$ guarantees the convexity of $D_\alpha$ and $M_\alpha$ along weak geodesics in the larger $\mathcal{E}^1$ space (cf. \cite{B15,BB17,BDL}).

\begin{proof}
$(1)$ follows directly from Jensen's inequality. (2) and (3) can be proved using the viewpoint of Legendre transform; see \cite[Lemma 2.15]{BBJ18} and also see the proof of Proposition \ref{prop:delta-A-Mabuchi}. (4) and (5) essentially follow from the variational principle of Darvas--Rubinstein \cite{DR}; see \cite[Theorem 2.19]{BBJ18} for a general statement that implies (4) and (5).
\end{proof}

\section{Analytic $\delta$-invariant}
\label{sec:delta-A}

Based on the spirit of Ding \cite{D}, we define an analytic $\delta$-invariant in terms of the optimal Moser--Trudinger constant, which resembles very much to Tian's formulation of his $\alpha$-invariant \cite{T87}. An advantage of this definition is that no polarization is needed. Moreover, as we will see, this analytic $\delta$-invariant is naturally related to the greatest Ricci lower bound. In the literature this analytic invariant has been implicitly studied by many authors; see for instance \cite{Sze11,B13,SW} for related discussions.

Let $X$ be a compact K\"ahler manifold of dimension $n$, whose K\"ahler cone will be denoted by $\mathcal{K}(X)$. Let $\xi\in\mathcal{K}(X)$ be a K\"ahler class and fix some K\"ahler form $\omega\in 2\pi c_1(\xi)$. Let $\mathcal{H}_0(X,\omega)$ denotes the space of normalized K\"ahler potentials of $\omega$ (see \eqref{eq:normalized-kahler-potential}).

\begin{definition}
The analytic $\delta$-invariant $\delta^A(\xi)$ of the K\"ahler class $\xi$ is defined as
\begin{equation}
    \delta^A(\xi):=\sup\bigg\{\lambda>0\ \bigg|\ 
    \exists\ C_\lambda>0\text{ s.t. }
    \int_Xe^{-\lambda\varphi}\omega^n\leq C_\lambda e^{\lambda J_\omega(\varphi)},\ 
    \forall\varphi\in\mathcal{H}_0(X,\omega)
    \bigg\}.
\end{equation}
\end{definition}

Note that
$\delta^A(\xi)$ is clearly bounded from below by the $\alpha$-invariant of Tian (see \eqref{eq:def-alpha}), and the definition does not depend on the choice of $\omega$.
When $\xi=c_1(L)$ is polarized by some ample $\QQ$-line bundle $L$, we also write
\begin{equation}
    \delta^A(L):=\delta^A(\xi).
\end{equation}
Observe that, $\delta^A(\cdot)$ satisfies the following scaling property: for any $\lambda>0$,
\begin{equation}
\label{eq:scaling-delt-A}
 \delta^A(\lambda\xi)=\lambda^{-1}\delta^A(\xi).
\end{equation}

The formulation of $\delta^A(\xi)$ is easily seen to be equivalent to the coercivity of certain twisted Ding (and hence Mabuchi) energy.
\begin{proposition}
\label{prop:delta-A-Mabuchi}
We have
$$
\delta^A(\xi)=\sup\bigg\{\lambda>0\bigg|
    \exists\ C_\lambda>0\text{ s.t. }
    \frac{1}{V}\int_X\log\frac{\omega_\varphi^n}{\omega^n}\omega_\varphi^n\geq\lambda(I_\omega-J_\omega)(\varphi)-C_\lambda,
    \forall\varphi\in\mathcal{H}(X,\omega)
    \bigg\}.
$$
\end{proposition}

\begin{proof}
This is essentially \cite[Proposition 4.11]{BBEGZ}. Denote the right hand side by $\delta^\prime(\xi)$.

We first show $\delta^A(
\xi)\geq\delta^\prime(\xi)$. Fix any $\lambda\in(0,\delta^\prime(\xi))$. By Calabi-Yau theorem, there is a unique $\phi_\varphi\in\mathcal{H}_0(X,\omega)$ associated to each $\varphi\in\mathcal{H}_0(X,\omega)$ such that $\omega^n_{\phi_\varphi}=e^{c-\lambda\varphi}\omega^n$ for some normalizing constant $c\in\RR$. Then we have
\begin{equation*}
    \begin{split}
        \log\bigg(\frac{1}{V}\int_Xe^{-\lambda\varphi}\omega^n\bigg)&=\log\bigg(\frac{1}{V}\int_Xe^{-\lambda\varphi}\frac{\omega^n}{\omega^n_{\phi_\varphi}}\omega^n_{\phi_\varphi}\bigg)\\
        &=\frac{-\lambda}{V}\int_X\varphi\omega^n_{\phi_\varphi}-\frac{1}{V}\int_X\log\frac{\omega^n_{\phi_\varphi}}{\omega^n}\omega^n_{\phi_\varphi}\\
        &\leq C_\lambda-\lambda(I_\omega-J_\omega)(\phi_\varphi)-\frac{\lambda}{V}\int_X\varphi\omega^n_{\phi_\varphi}\\
        &\leq C^\prime_\lambda+\lambda J_\omega(\varphi),\ \forall\varphi\in\mathcal{H}_0(X,\omega).\\
    \end{split}
\end{equation*}
In the last inequality we used \eqref{eq:J-cocycle} and \eqref{eq:int-vp>-C}. Thus we have $\delta^A(\xi)\geq\delta^\prime(\xi)$.

Now we show $\delta^A(\xi)\leq\delta^\prime(\xi)$, which is an easy consequence of Jensen's inequality. Indeed, Fix any $\lambda\in(0,\delta^A(\xi))$. Then there exists $C_\lambda>0$ such that
\begin{equation*}
    \begin{split}
        C_\lambda+\lambda J_\omega(\varphi)&\geq\log\bigg(\frac{1}{V}\int_Xe^{-\lambda\varphi}\omega^n\bigg)\\
        &=\log\bigg(\frac{1}{V}\int_Xe^{-\lambda\varphi}\frac{\omega^n}{\omega_\varphi^n}\omega^n_\varphi\bigg)\\
        &\geq-\frac{\lambda}{V}\int_X\varphi\omega^n_\varphi-\frac{1}{V}\int_X\log\frac{\omega^n_\varphi}{\omega^n}\omega^n_\varphi,\ \forall\varphi\in\mathcal{H}_0(X,\omega).\\
    \end{split}
\end{equation*}
So by \eqref{eq:int-vp>-C},
$$
\frac{1}{V}\int_X\log\frac{\omega^n_\varphi}{\omega^n}\omega^n_\varphi\geq\lambda(I_\omega-J_\omega)(\varphi)-C^\prime_\lambda,\ \forall\varphi\in\mathcal{H}_0(X,\omega).
$$
Thus $\delta^\prime(\xi)\geq\delta^A(\xi)$ is proved.
\end{proof}

Now the next result is clear.
\begin{proposition}
\label{prop:delta-A>1=proper}
The following are equivalent.
\begin{enumerate}
    \item $\delta^A(\xi)>1$.
    \item For any smooth form $\alpha\in 2\pi (c_1(X)-\xi)$, $D_\alpha$ is coercive.
    \item For any smooth form $\alpha\in 2\pi (c_1(X)-\xi)$, $M_\alpha$ is coercive.
\end{enumerate}
\end{proposition}

\begin{proof}
For any smooth form $\alpha\in 2\pi (c_1(X)-\xi)$, the normalized Ricci potential $f_\alpha$ is a bounded function (recall \eqref{eq:def-f_alpha}). So by Proposition \ref{prop:delta-A-Mabuchi}, $\delta^A(\xi)>1$ is equivalent to $M_\alpha$ (and hence $D_\alpha$) being coercive.
\end{proof}
The following result is an analytic version of Theorem \ref{thm:YTD}.
\begin{proposition} 
\label{prop:YTD-delta-A}
The analytic $\delta$-invariant has the following properties.
\begin{enumerate}
    \item Assume that $\delta^A(\xi)>1$, then for any $\alpha\in 2\pi (c_1(X)-\xi)$, there exists $\omega\in2\pi\xi$ such that
    $\Ric(\omega)=\omega+\alpha$.
    
    \item Assume that there exists an semi-positive smooth form $\alpha\in 2\pi (c_1(X)-\xi)$. If there exists $\omega\in2\pi\xi$ (resp. a unique $\omega\in2\pi\xi$) such that
    $\Ric(\omega)=\omega+\alpha$, then $\delta^A(\xi)\geq1$ (resp. $\delta^A(\xi)>1$).
    
\end{enumerate}

\end{proposition}

\begin{proof}
Part $(1)$ follows directly from Proposition \ref{prop:delta-A>1=proper} and Theorem \ref{thm:D-M-tKE}.(3). For part $(2)$, the existence of a tKE solution implies that $M_\alpha$ is bounded from below by Theorem \ref{thm:D-M-tKE}.(4), which implies that $\delta^A(\xi)\geq1$ in the view of Proposition \ref{prop:delta-A-Mabuchi}. If this solution is unique, the $M_\alpha$ is further coercive by Theorem \ref{thm:D-M-tKE}.(5), so $\delta^A(\xi)>1$ by Proposition \ref{prop:delta-A>1=proper}.
\end{proof}

Now we show that $\delta^A(\xi)$ captures the greatest Ricci lower bound $\beta(\xi)$ of the K\"ahler class $\xi$. Recall here that $s(\xi)$ defined in \eqref{eq:def-S-xi} denotes the K\"ahler threshold.
\begin{proposition}
\label{prop:beta-xi=s-delta-A}
For any K\"ahler class $\xi$, we have
\begin{equation*}
    \beta(\xi)=\min\{s(\xi),\delta^A(\xi)\}.
\end{equation*}
\end{proposition}

\begin{proof}
We borrow the argument from \cite[Section 7.3]{BBJ18}.
The proof depends on the sign of $s(\xi)$. 
First consider the case $s(\xi)\leq0$. Then for any $s<s(\xi)$, pick an smooth positive form $\alpha\in 2\pi(c_1(X)-s\xi)$. Then by \cite{BBGZ}, there always exists $\omega\in2\pi\xi$ solving $\Ric(\omega)=s\omega+\alpha$. So we have $\beta(\xi)\geq s(\xi)$. The direction $\beta(\xi)\leq s(\xi)$ is trivial. So we obtain $\beta(\xi)=s(\xi)$.
Now assume that $s(\xi)>0$. Then for any $0<s<\min\{s(\xi),\delta^A(\xi)\}$, we have $\delta^A(s\xi)=\delta^A(\xi)/s>1$. Pick any positive form $\alpha\in2\pi(c_1(X)-s\xi)$, then by Proposition \ref{prop:YTD-delta-A}.(1) we can find $\omega\in2\pi\xi$ such that $\Ric(\omega)=s\omega+\alpha>s\omega$. So we deduce that $\beta(\xi)\geq\min\{s(\xi),\delta^A(\xi)\}$. Now for any $0<\beta<\beta(\xi)$, there exist $\omega\in2\pi\xi$ and a positive form $\alpha\in2\pi(c_1(X)-\beta\xi)$ such that $\Ric(\omega)=\beta\omega+\alpha$. Then by Proposition \ref{prop:YTD-delta-A}.(2), $\delta^A(\beta\xi)\geq1$, so that $\beta\leq\delta^A(\xi)$. Combining with $\beta\leq s(\xi)$, we thus have $\beta(\xi)\leq\min\{s(\xi),\delta^A(\xi)\}$. This completes the proof.

\end{proof}

\begin{remark}
\rm{
When $\xi=c_1(X)$, Proposition \ref{prop:beta-xi=s-delta-A} also appeared as \cite[Corollary 2.3]{SW}.
}
\end{remark}

Now by Theorem \ref{thm:BBJ-beta-L} and Proposition \ref{prop:beta-xi=s-delta-A}, we have the following consequence.

\begin{corollary}
\label{cor:s-delta=s-delta-A-Q}
For any ample $\QQ$-line bundle $L$, one has
$$
\beta(L)=\min\{s(L),\delta^A(L)\}=\min\{s(L),\delta(L)\}.
$$
\end{corollary}

Using the continuity of $\delta$ and $\delta^A$ (which will be shown in the next section), we see that the above equality holds for ample $\RR$-line bundles as well.

We also have the following relation.
\begin{proposition}
\label{prop:delta-A<=delta-Q}
For any ample $\QQ$-line bundle $L$, we have
$$
\delta^A(L)\leq\delta(L).
$$
\end{proposition}
\begin{proof}
Pick any $\lambda\in(0,\delta^A(L))$, it suffices to show $\delta(L)>\lambda$. By rescaling $L$, we might as well assume that $\lambda=1$. So that $\delta^A(L)>1$ and hence for any $\alpha\in 2\pi (c_1(X)-\xi)$, $D_\alpha$ is coercive. Then by \cite[Theorem 5.1]{BBJ18}, $L$ is uniformly Ding stable, i.e. $\delta(L)>1$.
\end{proof}

It is expected that $\delta^A(L)$ agrees with $\delta(L)$. 
If this is true, then by Proposition \ref{prop:YTD-delta-A}, one could substantially improve our main result (Theorem \ref{thm:YTD}) so that the appearance of non-semipositive twist terms can be allowed.

\section{Continuity}
\label{sec:conti}
So far, we have introduced $\delta(\cdot)$, $\delta^A(\cdot)$ and $\beta(\cdot)$. The purpose of this section is to show that all these thresholds vary continuously on their domains.

\subsection{Proof of Theorem \ref{thm:delta-is-continuous}}
To prove the continuity of $\delta(\cdot)$, we use the strategy of Dervan \cite{D15,D16}. The key point is to establish the following comparison principle.

\begin{proposition}
\label{prop:delta-comparison}
There exists $\varepsilon_0$ only depending $n$ such that the following holds. For any big $\RR$-line bundle $L$ and any $\varepsilon\in(0,\varepsilon_0)$, let $L_\varepsilon$ be any small perturbation of $L$ such that
$$
\text{both }(1+\varepsilon)L-L_\varepsilon\text{ and }L_\varepsilon-(1-\varepsilon)L\text{ are big.}
$$
Then we have
$$
\delta(L+\varepsilon L_\varepsilon)\leq\delta(L)\leq\delta(L-\varepsilon L_\varepsilon).
$$
\end{proposition}

\begin{proof}
We only prove $\delta(L+\varepsilon L_\varepsilon)\leq\delta(L)$, since the other part follows in a similar manner.
Let $F$ be any prime divisor over $X$. It suffices to show
$$
S_{L+\varepsilon L_\varepsilon}(F)\geq S_L(F).
$$
To this end, we calculate as follows:
\begin{equation*}
    \begin{split}
        S_{L+\varepsilon L_\varepsilon}(F)&=\frac{1}{\vol(L+\varepsilon L_\varepsilon)}\int_0^\infty\vol(L+\varepsilon L_\varepsilon-tF)dt\\
        &\geq\frac{1}{\vol(L+(\varepsilon+\varepsilon^2)L)}\int_0^\infty\vol(L+(\varepsilon-\varepsilon^2)L-tF)dt\\
        &=\bigg(\frac{1+\varepsilon-\varepsilon^2}{1+\varepsilon+\varepsilon^2}\bigg)^n\cdot S_{(1+\varepsilon-\varepsilon^2)L}(F)\\
        &=\bigg(\frac{1+\varepsilon-\varepsilon^2}{1+\varepsilon+\varepsilon^2}\bigg)^n\cdot(1+\varepsilon-\varepsilon^2)\cdot S_{L}(F).\\
    \end{split}
\end{equation*}
Here we used the monotonicity of $\vol(\cdot)$ (recall \eqref{eq:vol-L+B<vol-L}).
By choosing $\varepsilon$ small enough we can arrange that
$$
\bigg(\frac{1+\varepsilon-\varepsilon^2}{1+\varepsilon+\varepsilon^2}\bigg)^n\cdot(1+\varepsilon-\varepsilon^2)\geq1.
$$
This completes the proof.
\end{proof}

\begin{theorem}(=Theorem \ref{thm:delta-is-continuous})
The delta invariant $\delta(\cdot)$ is a continuous function on the big cone.
\end{theorem}

\begin{proof}

Let $L$ be a big $\RR$-line bundle. Fix any auxiliary $\RR$-line bundle $S\in N^1(X)_\RR$. We need to show that, for any small $\varepsilon>0$, there exists $\gamma>0$ such that
$$
(1-\varepsilon)\delta(L)\leq\delta(L+\gamma S)\leq(1+\varepsilon)\delta(L).
$$
Here $L+\gamma S$ is always assumed to be big (by choosing $\gamma$ sufficiently small). Notice that for any $\varepsilon>0$, we can write
$$
L+\gamma S=\frac{1}{1+\varepsilon}\bigg(L+\varepsilon\big(L+\frac{(1+\varepsilon)\gamma}{\varepsilon}S\big)\bigg).
$$
Put
$$
L_\varepsilon:=L+\frac{(1+\varepsilon)\gamma}{\varepsilon}S.
$$
Then by choosing $\gamma$ small enough, we can assume that
$$
\text{both }(1+\varepsilon)L-L_\varepsilon\text{ and }L_\varepsilon-(1-\varepsilon)L\text{ are big.}
$$
So from the scaling property of $\delta(\cdot)$ and Proposition \ref{prop:delta-comparison}, it follows that
$$
\delta(L+\gamma S)=(1+\varepsilon)\delta(L+\varepsilon L_\varepsilon)\leq(1+\varepsilon)\delta(L).
$$
We can also write
$$
L+\gamma S=\frac{1}{1-\varepsilon}\bigg(L-\varepsilon\big(L-\frac{(1-\varepsilon)\gamma}{\varepsilon}S\big)\bigg).
$$
Then a similar treatment as above yields
$$
\delta(L+\gamma S)\geq(1-\varepsilon)\delta(L).
$$
In conclusion, for any small $\varepsilon>0$, by choosing $\gamma$ to be sufficiently small, we have
$$
(1-\varepsilon)\delta(L)\leq\delta(L+\gamma S)\leq(1+\varepsilon)\delta(L).
$$
This completes the proof.
\end{proof}

\subsection{Proof of Theorem \ref{thm:delta-A-continuous}} Now let us establish the continuity of $\delta^A(\cdot)$ using the above strategy. The key result is the following comparison principle.

\begin{proposition}
\label{prop:delta-A-comparison}
There exists $\varepsilon_0$ only depending $n$ such that the following holds. For any K\"ahler class $\xi$ and any $\varepsilon\in(0,\varepsilon_0)$, let $\xi_\varepsilon$ be any small perturbation of $\xi$ such that
there are two K\"ahler forms $\omega\in2\pi\xi$ and $\omega_\varepsilon\in2\pi\xi_\varepsilon$ satisfying
$$
(1-\varepsilon)\omega\leq\omega_\varepsilon\leq(1+\varepsilon)\omega.
$$
Then we have
$$
\delta^A(\xi+\varepsilon \xi_\varepsilon)\leq\delta^A(\xi)\leq\delta^A(\xi-\varepsilon \xi_\varepsilon).
$$
\end{proposition}

\begin{proof}
We only show $\delta^A(\xi+\varepsilon \xi_\varepsilon)\leq\delta^A(\xi)$, since the proof for the other part is similar.
For any $\varphi\in\mathcal{H}(X,\omega)\subseteq\mathcal{H}(X,\omega+\varepsilon\omega_\varepsilon)$ and $\lambda>0$, it is clear that,
$$
\int_Xe^{-\lambda\varphi}(\omega+\varepsilon\omega_\varepsilon)^n\geq\int_Xe^{-\lambda\varphi}\omega^n.
$$
So it suffices to show
$$
J_{\omega+\varepsilon\omega_\varepsilon}(\varphi)\leq J_\omega(\varphi),\ \forall\varphi\in\mathcal{H}(X,\omega).
$$
For this, we compute
\begin{equation*}
    \begin{split}
        J_{\omega+\varepsilon\omega_\varepsilon}(\varphi)&=\frac{1}{\int_X(\omega+\varepsilon\omega_\varepsilon)^n}\int_0^1\int_X\varphi\big((\omega+\varepsilon\omega_\varepsilon)^n-(\omega+\varepsilon\omega_\varepsilon+\idd s\varphi)^n\big)ds\\
        &=\frac{1}{\int_X(\omega+\varepsilon\omega_\varepsilon)^n}\int_0^1\int_X\sum_{j=0}^ns\sqrt{-1}\partial\varphi\wedge\bar{\partial}\varphi\wedge(\omega+\varepsilon\omega_\varepsilon)^{n-1-j}\wedge(\omega+\varepsilon\omega_\varepsilon+\idd s\varphi)^jds\\
        &\leq\frac{(1+\varepsilon+\varepsilon^2)^{n-1}}{\int_X(\omega+(\varepsilon-\varepsilon^2)\omega)^n}\int_0^1\int_X\sum_{j=0}^ns\sqrt{-1}\partial\varphi\wedge\bar{\partial}\varphi\wedge\omega^{n-1-j}\wedge(\omega+\idd s\varphi/(1+\varepsilon+\varepsilon^2))^jds\\
        &=\bigg(\frac{1+\varepsilon+\varepsilon^2}{1+\varepsilon-\varepsilon^2}\bigg)^n\cdot\frac{1}{\int_X\omega^n}\int_0^1\int_X\varphi\big(\omega^n-(\omega+\idd s\varphi/(1+\varepsilon+\varepsilon^2))^n\big)ds.\\
        &=\bigg(\frac{1+\varepsilon+\varepsilon^2}{1+\varepsilon-\varepsilon^2}\bigg)^n\cdot (1+\varepsilon+\varepsilon^2)\cdot J_\omega\bigg(\frac{\varphi}{1+\varepsilon+\varepsilon^2}\bigg)\\
        &\leq\bigg(\frac{1+\varepsilon+\varepsilon^2}{1+\varepsilon-\varepsilon^2}\bigg)^n\cdot\bigg(\frac{1}{1+\varepsilon+\varepsilon^2}\bigg)^{\frac{1}{n}}\cdot J_\omega(\varphi)\\
    \end{split}
\end{equation*}
We used Ding's inequality \eqref{eq:ding-ineq}  in the last step.
Now by choosing $\varepsilon$ to be sufficiently small, we can arrange that
$$
\bigg(\frac{1+\varepsilon+\varepsilon^2}{1+\varepsilon-\varepsilon^2}\bigg)^n\cdot\bigg(\frac{1}{1+\varepsilon+\varepsilon^2}\bigg)^\frac{1}{n}\leq1.
$$
This completes the proof.
\end{proof}

\begin{theorem}
(=Theorem \ref{thm:delta-A-continuous})
The analytic $\delta$-invariant
    $\delta^A(\cdot)$ is a continuous function on the K\"ahler cone.
\end{theorem}

\begin{proof}Now with \eqref{eq:scaling-delt-A} and Proposition \ref{prop:delta-A-comparison}, the proof is almost the same as the one for Theorem \ref{thm:delta-is-continuous}. So we omit it.
\end{proof}

As a consequence we obtain the continuity of $\beta(\cdot)$, which improves \cite[Lemma 4.3]{Z20} in the author's recent work. 
\begin{theorem}
The greatest Ricci lower bound $\beta(\cdot)$ is a continuous function on the K\"ahler cone.
\end{theorem}

\begin{proof}
Recall
$
\beta(\xi)=\min\{s(\xi),\delta^A(\xi)\}
$
(see Proposition \ref{prop:beta-xi=s-delta-A}).
Since both $s(\cdot)$ are $\delta^A(\cdot)$ are continuous on the K\"ahler cone, so is $\beta(\cdot)$.
\end{proof}

By the continuity of $\delta(\cdot)$ and $\delta^A(\cdot)$, we can extend Corollary \ref{cor:s-delta=s-delta-A-Q} to $\RR$-line bundles.
\begin{corollary}
\label{cor:s-delta=s-delta-A-R}
For any ample $\RR$-line bundle $\xi$, one has
$$
\beta(\xi)=\min\{s(\xi),\delta^A(\xi)\}=\min\{s(\xi),\delta(\xi)\}.
$$
\end{corollary}

We can also extend Proposition \ref{prop:delta-A<=delta-Q} by continuity.

\begin{proposition}
\label{prop:delta-A<=delta-R}
For any ample $\RR$-line bundle $\xi$, one has
$$
\delta^A(\xi)\leq\delta(\xi).
$$
\end{proposition}

\section{Existence of twisted K\"ahler--Einstein metrics}
\label{sec:YTD}

This section is devoted to proving the following result. 
\begin{theorem}[=Theorem \ref{thm:YTD}]
\label{thm:YTD'}
Let $\xi$ be an ample $\RR$-line bundle on $X$. Fix any smooth form $\alpha\in2\pi(c_1(X)-c_1(\xi))$ and assume that $\alpha\geq0$. Then we have
\begin{enumerate}
    \item If $\delta(\xi)>1$, then there exists $\omega\in2\pi c_1(\xi)$ such that $\Ric(\omega)=\omega+\alpha$.
    
    \item If there exists $\omega\in2\pi c_1(\xi)$ (resp. a unique $\omega\in2\pi c_1(\xi)$) such that $\Ric(\omega)=\omega+\alpha$, then $\delta(\xi)\geq1$ (resp. $\delta(\xi)>1$).
\end{enumerate}
\end{theorem}

We will deal with part (2) first, since it is easier.
\begin{proof}[Proof of Theorem \ref{thm:YTD'}.(2)]
Assume that there exists $\omega\in2\pi c_1(\xi)$ such that $\Ric(\omega)=\omega+\alpha$. Then by Proposition \ref{prop:YTD-delta-A}.(2), we have $\delta^A(\xi)\geq1$. If moreover $\omega$ is unique, then $\delta^A(\xi)>1$. Thus the result follows from the inequality $\delta(\xi)\geq\delta^A(\xi)$ (Proposition \ref{prop:delta-A<=delta-R}).



\end{proof}

Now we turn to the first part of Theorem \ref{thm:YTD'}. The rough idea is as follows. Assume that $\delta(\xi)>1$ and choose a sequence of ample $\QQ$-line bundles $L_i$ approximating $\xi$. We can also assume that there is a sequence of smooth semi-positive forms $\alpha_i\in2\pi(c_1(X)-c_1(L_i))$ converging smoothly to $\alpha\in2\pi(c_1(X)-c_1(\xi))$. By the continuity of $\delta$-invariant, we have $\delta(L_i)>1$. So Theorem \ref{thm:BBJ-YTD} gives $\omega_i\in2\pi c_1(L_i)$ such that $\Ric(\omega_i)=\omega_i+\alpha_i$. We wish to show that $\omega_i$ converges smoothly to the desired tKE metric in $2\pi c_1(\xi)$. To make this argument work, the key point is to establish a uniform $C^k$-bound (for any $k\geq0$) for the sequence $\{\omega_i\}$. However this is not trivial at all. Essentially, what we need is a uniform control of the twisted 
Mabuchi functionals $M_{\alpha_i}$ as $i\rightarrow\infty$. More precisely, we need the following quantitative estimate.

\begin{proposition}
\label{prop:delta-A>1+e}
Let $L$ be an ample $\QQ$-line bundle.
Assume that there is a semipositive smooth form $\alpha\in2\pi(c_1(X)-c_1(L))$ and that $\delta(L)\geq1+\varepsilon$ for some small $\varepsilon>0$. Then there exists $\varepsilon^\prime>0$ only depending on $n,\varepsilon$ such that
$\delta^A(L)\geq1+\varepsilon^\prime$.
\end{proposition}

\begin{proof}
This mainly follows from the argument in \cite[v1]{BBJ18}. We sketch the proof for the reader's convenience. Fix any K\"ahler form $\omega\in2\pi c_1(L)$ and consider the following functionals:
\begin{equation}
\label{eq:def-E}
    E(\varphi):=\frac{1}{V}\int_X\varphi\omega^n-J(\varphi)=\frac{1}{(n+1)V}\int_X\sum_{i=0}^n\varphi\omega^i\wedge\omega^{n-i}_\varphi,\ \varphi\in\mathcal{H}(X,\omega).
\end{equation}
\begin{equation}
    L_\alpha(\varphi):=-\log\bigg(\frac{1}{V}\int_Xe^{f_\alpha-\varphi}\omega^n\bigg),\ \varphi\in\mathcal{H}(X,\omega).
\end{equation}
Then the $\alpha$-twisted Ding functional $D_\alpha$ can be written as
$$
D_\alpha=L_\alpha-E.
$$
Also consider the $\alpha$-twisted Mabuchi functional $M_\alpha$. By Proposition \ref{prop:delta-A-Mabuchi}, our goal is to find $\varepsilon^\prime>0$ and $C>0$ such that
$$
M_\alpha(\varphi)\geq\varepsilon^\prime (I_\omega-J_\omega)(\varphi)-C,\ \forall\varphi\in\mathcal{H}(X,\omega).
$$
For this, we will argue by contradiction. Assume that for some $\varepsilon^\prime>0$, there exists a sequence $\phi_j\in\mathcal{H}_0(X,\omega)$ such that
\begin{equation}
\label{eq:M<e(I-J)-j}
    M_\alpha(\phi_j)\leq\varepsilon^\prime(I_\omega-J_\omega)(\phi_j)-j.
\end{equation}
Now we need to work in the larger space $\mathcal{E}^1(X,\omega)$, where all the energy functionals in this paper can be defined. Most importantly, one can consider the geodesic segment $(\phi_{j,t})_{0\leq t\leq T_j}$ from $0$ to $\phi_j$. Note that $E$ is affine along geodesics, so we can assume $E(\phi_{j,t})=-t$ and moreover, we have $\sup\phi_{j,t}=0$. Using \eqref{eq:M<e(I-J)-j}, one can further extract a geodesic ray $(\phi_t)_{t\geq0}$, so that $E(\phi_t)=-t$. And also by convexity of $M_\alpha$ (together with Theorem \ref{thm:D-M-tKE}.(1) and \eqref{eq:I-J-lower-upper-J}), we have
\begin{equation}
    \label{eq:D<M<net}
    D_\alpha(\phi_t)\leq M_\alpha(\phi_t)\leq n\varepsilon^\prime t.
\end{equation}
Now as in \cite[v1, \S 3.2]{BBJ18}, we can approximate $\phi_t$ by a sequence of geodesic rays $\{\phi_{m,t}\}$ arising from test configurations, which corresponds to a sequence $\{\varphi_m\}$ in the non-Archimedean world $\mathcal{H}^{NA}$. More precisely, we have (the twist term $\alpha$ will not play any role in NA functionals so we drop it)
$$
L^{NA}(\varphi_m):=\lim_{t\rightarrow\infty}\frac{1}{t}L_\alpha(\phi_{m,t}),
$$
and (see \cite[v1, (3.1)]{BBJ18})
$$
E^{NA}(\varphi_m)=-J^{NA}(\varphi_m):=\lim_{t\rightarrow\infty}\frac{1}{t}E(\phi_{m,t})\geq-1.
$$
The key identity is \cite[v1, (3.2)]{BBJ18} (see also \cite[Lemma 5.7]{BBJ18}):
$$
\lim_{m\rightarrow\infty}L^{NA}(\varphi_m)=\lim_{t\rightarrow\infty}\frac{1}{t}L_\alpha(\phi_{t}).
$$
Then by
\eqref{eq:D<M<net}
we have
$$
\lim_{m\rightarrow\infty}L^{NA}(\varphi_m)=\lim_{t\rightarrow\infty}\frac{1}{t}\big(D_\alpha(\phi_t)+E(\phi_t)\big)\leq n\varepsilon^\prime-1.$$
On the other hand, by the proof of \cite[Theorem 7.3]{BBJ18}, the assumption $\delta(L)\geq1+\varepsilon$ implies that $D^{NA}=L^{NA}-E^{NA}\geq(1-(1+\varepsilon)^{-1/n})J^{NA}$. So we have
$$
L^{NA}(\varphi_m)\geq-(1+\varepsilon)^{-1/n}J^{NA}(\varphi_m)\geq-(1+\varepsilon)^{-1/n}.
$$
Then we would get a contradiction as soon as
$$
\varepsilon^\prime<\frac{1-(1+\varepsilon)^{-1/n}}{n}.
$$
Thus we have shown $\delta^A(L)\geq1+\frac{1-(1+\varepsilon)^{-1/n}}{n}$, as desired.
\end{proof}

Now we are able to conclude the first part of Theorem \ref{thm:YTD'}.

\begin{proof}
[Proof of Theorem \ref{thm:YTD'}.(1)]
Assume that $\delta(\xi)\geq1+\varepsilon$ for some small $\varepsilon>0$. Pick a sequence of ample $\QQ$-line bundles $L_i\rightarrow\xi$ with $s(L_i)>1$. By the continuity of $\delta$, we can assume that $\delta(L_i)\geq1+\varepsilon/2$. Then by Proposition \ref{prop:delta-A>1+e}, we can find $\varepsilon^\prime>0$ such that $\delta^A(L_i)\geq1+\varepsilon^\prime$ for all $i$. Now by the continuity of $\delta^A$, we get $\delta^A(\xi)\geq1+\varepsilon^\prime$. Thus the assertion follows from Proposition \ref{prop:YTD-delta-A}.(1).
\end{proof}
The proof of Theorem \ref{thm:YTD} is complete.

As a by product of the above argument, we obtain the following result.

\begin{theorem}[=Theorem \ref{thm:YTD-any-alpha}]
\label{thm:YTD-any-alpha'}
Let $X$ be a Fano manifold. Let $\xi$ be an ample $\RR$-line bundle on $X$. Then we have
\begin{enumerate}
    \item If $\delta(\xi)\leq s(\xi)$, then for any $\delta\in(0,\delta(\xi))$ and any smooth form $\alpha\in2\pi(c_1(X)-\delta\xi)$, there exists $\omega\in2\pi c_1(\xi)$ such that $\Ric(\omega)=\delta\omega+\alpha$. 
    
    \item If $\delta(\xi)>s(\xi)$, then for any smooth form $\alpha\in2\pi(c_1(X)-s(\xi)\xi)$, there exists $\omega\in2\pi c_1(\xi)$ such that
    $\Ric(\omega)=s(\xi)\omega+\alpha$.
\end{enumerate}
\end{theorem}

\begin{proof}
For the first part, we have $\delta^A(\xi)=\delta(\xi)$ by Corollary \ref{cor:s-delta=s-delta-A-R}. Then the assertion follows from Proposition \ref{prop:YTD-delta-A}.(1) (by suitably scaling $\xi$). Now for the second part, it is enough to notice that the assumption $\delta(\xi)>s(\xi)$ implies that $\delta^A(\xi)>s(\xi)$. Indeed, assume that $\delta(\xi)\geq(1+\epsilon)s(\xi)$ for some $\varepsilon>0$. Consider $\xi^\prime:=(s(\xi)-\eta)\xi$, where $\eta>0$ is any sufficiently small number. Then one can make sure that $s(\xi^\prime)>1$ and $\delta(\xi^\prime)\geq1+\varepsilon/2$. So Proposition \ref{prop:delta-A>1+e} together with the continuity of $\delta$ and $\delta^A$ imply that $\delta^A(\xi^\prime)\geq1+\varepsilon^\prime$ for some $\varepsilon^\prime>0$ depending only on $n$ and $\varepsilon$. Thus $\delta^A(\xi)\geq(1+\varepsilon^\prime)(s(\xi)-\eta)$. Letting $\eta\rightarrow0$, we get $\delta^A(\xi)\geq(1+\varepsilon^\prime)s(\xi)>s(\xi)$, as desired. Then we conclude by Proposition \ref{prop:YTD-delta-A}.(1).
\end{proof}

\section{Further discussions}
\label{sec:delt-delta-A}
\subsection{More on the analytic $\delta$-invariant}
 It is reasonable to believe that
 $$
 \delta^A(L)=\delta(L)
 $$
 for any ample $\QQ$-line bundles. However it seems that the methods in this paper cannot provide a straightforward way to prove this. One obstacle comes from the argument of Proposition \ref{prop:delta-A>1+e} (i.e., the variational approach in \cite{BBJ18}), which crucially relies on the limiting behavior of Ding functionals. This unfortunately prohibits us from getting the optimal lower bound for $\delta^A$. More discussions regarding this problem will appear in a separated paper.
 We also refer the reader to \cite{L20} for some recent progress in this direction.
 

In the following, let us collect more properties of $\delta^A$.
Let $X$ be a compact K\"ahler manifold. Let $\omega\in2\pi\xi$ be a K\"ahler form.
We first recall the $\alpha$-invariant $\alpha(\xi)$ of Tian \cite{T87}:
\begin{equation}
\label{eq:def-alpha}
    \alpha(\xi):=\sup\bigg\{\lambda>0\ \bigg|
    \exists\ C_\lambda>0
    \text{ such that }
    \int_Xe^{-\lambda\varphi}\omega^n\leq C_\lambda
    \text{ for any }\varphi\in\mathcal{H}_0(X,\omega)
    \bigg\}.
\end{equation}
It is shown by Tian that one always has $\alpha(\xi)>0$. When $\xi=c_1(L)$, the above definition agrees with the algebraic definition using log canonical threshold (cf. \cite[Appendix]{CS}).
Also recall that the continuity of $\alpha$ is proved by Dervan \cite{D16}.
The $\alpha$-invariant plays significant roles in the study of canonical metrics. The following result explains the reason.
 
 \begin{proposition}
 For any K\"ahler class $\xi$, one has
 $$
 \delta^A(\xi)\geq\frac{n+1}{n}\alpha(\xi).
 $$
 \end{proposition}
 
 \begin{proof}
 This is well known and follows easily from Jensen's inequality (see \cite[Theorem 7.13]{T00}).
 Indeed, pick any $\lambda\in(0,\alpha(\xi))$. Then for some $C_\lambda>0$, we have
 \begin{equation*}
     \begin{split}
         C_\lambda\geq\log\bigg(\frac{1}{V}\int_Xe^{-\lambda\varphi}\omega^n\bigg)&=\log\bigg(\frac{1}{V}\int_Xe^{-\lambda\varphi}\frac{\omega^n}{\omega^n_\varphi}\omega_\varphi^n\bigg)\\
         &\geq-\frac{\lambda}{V}\int_X\varphi\omega^n_\varphi+\frac{1}{V}\int_X\log\frac{\omega^n}{\omega_\varphi^n}\omega^n_\varphi\\
         &\geq\frac{\lambda}{V}I_\omega(\varphi)-\frac{1}{V}\int_X\log\frac{\omega^n_\varphi}{\omega^n}\omega^n_\varphi\\
         &\geq\frac{(n+1)\lambda}{n}(I_\omega-J_\omega)(\varphi)-\frac{1}{V}\int_X\log\frac{\omega^n_\varphi}{\omega^n}\omega^n_\varphi,\ \forall\varphi\in\mathcal{H}_0(X,\omega).\\
     \end{split}
 \end{equation*}
 So we have
 $$
 \frac{1}{V}\int_X\log\frac{\omega^n_\varphi}{\omega^n}\omega^n_\varphi\geq\frac{(n+1)\lambda}{n}(I_\omega-J_\omega)(\varphi)-C_\lambda.
 $$
 So Proposition \ref{prop:delta-A-Mabuchi} implies that $\delta^A(\xi)\geq\frac{(n+1)\lambda}{n}$, hence finishing the proof.
 \end{proof}
 
 For ample $\RR$-line bundles, one can also bound $\delta^A$ from above using $\alpha$-invariant.
 \begin{proposition}
 For any ample $\RR$-line bundle $\xi$,  we have
 $$
 \delta^A(\xi)\leq\delta(\xi)\leq(n+1)\alpha(\xi).
 $$
 \end{proposition}
 
 \begin{proof}
 By \cite[Theorem A]{BJ17} and the continuity of $\alpha$ and $\delta$, we have 
 $$\delta(\xi)\leq(n+1)\alpha(\xi)$$
 for any ample $\RR$-line bundle $\xi$.
 So the assertion follows from Proposition \ref{prop:delta-A<=delta-R}.
 \end{proof}
 
 The following result gives a Bishop type volume estimate for ample $\RR$-line bundles. See also \cite{Z20} for related discussions.
 \begin{proposition}
 For any ample $\RR$-line bundle $\xi$,  we have
 $$
 \delta^A(\xi)^n\cdot\vol(\xi)\leq(n+1)^n.
 $$
 \end{proposition}
 
 \begin{proof}
 This follows from \cite[Theorem D]{BJ17}, Proposition \ref{prop:delta-A<=delta-R} and the continuity of $\delta^A(\cdot)$ and $\vol(\cdot)$.
 \end{proof}
 
 So we are led to the following questions. 
  \begin{question}[Lower bound of $\alpha$-invariant]
Do we have
 $$
\alpha(\xi)\geq\frac{\delta^A(\xi)}{n+1}
 $$
for any K\"ahler class $\xi$? 
 \end{question}
 
 \begin{question}[Volume comparison for K\"ahler classes]
Do we have
 $$
 \delta^A(\xi)^n\cdot[\xi]^n\leq(n+1)^n
 $$
 for any K\"ahler class $\xi$? 
 \end{question}
 To answer these questions, a suitable definition of 'Newton--Okounkov bodies' for K\"ahler classes would probably help; see \cite{BB17a} for some related discussions. 
 
\subsection{Constant scalar curvature K\"ahler metrics}
\label{sec:cscK}
Analytic $\delta$-invariant also play a role in finding the constant scalar curvature K\"ahler (cscK) metrics. Indeed, as we have seen in Proposition \ref{prop:delta-A-Mabuchi}, $\delta^A(\xi)$ can be characterized as the optimal coercivity constant for the \emph{entropy}
\begin{equation}
    H(\varphi):=\frac{1}{V}\int_X\log\frac{\omega_\varphi^n}{\omega^n}\omega^n_\varphi.
\end{equation}
Now recall that, in the view of \cite{CC18}, to find a cscK metric in the K\"ahler class $2\pi\xi$, one needs to check the coercivity of the following K-energy:
\begin{equation}
    K(\varphi):=H(\varphi)+\mathcal{J}(\varphi),
\end{equation}
where
\begin{equation}
    \mathcal{J}(\varphi):=n\frac{(-K_X)\cdot\xi^{n-1}}{\xi^n}E(\varphi)-\frac{1}{V}\int_X\varphi\Ric(\omega)\wedge\sum_{i=0}^{n-1}\omega^i\wedge\omega^{n-1-i}_\varphi
\end{equation}
 is the \emph{energy} part of $K(\cdot)$. Regarding $\mathcal{J}(\cdot)$, one can also consider its coercivity threshold (cf. \cite{D19})
 \begin{equation}
     \gamma(\xi):=\sup\bigg\{\epsilon\in\RR\ \bigg|\ \exists\  C_\epsilon>0\ \text{s.t. }\mathcal{J}(\varphi)\geq\epsilon(I-J)(\varphi)-C_\epsilon,\ \forall\varphi\in\mathcal{H}(X,\omega)\bigg\}.
 \end{equation}
 Thus by \cite{CC18}, there exists a unique cscK metric in $2\pi\xi$ if
 \begin{equation}
     \delta^A(\xi)+\gamma(\xi)>0.
 \end{equation}

So a basic question is to understand how small $\gamma(\xi)$ can be, which is intimately related to the J-stability literature. For precise estimates of the $\gamma$-invariant, we refer the reader to \cite{D19}. Below we state a result that is probably more friendly for calculations. Following the proof of \cite[Theorem 1.1]{LSY} (cf. also \cite{W06}), one can easily obtain the following criterion for the existence of cscK metrics.
 
 \begin{corollary}
 Let $X$ be a compact K\"ahler manifold of dimension $n$. Let $\xi$ be a K\"ahler class on $X$. Assume that
 \begin{enumerate}
     \item 
     $
 K_X+\delta^A(\xi)\xi\text{ is K\"ahler}
 $,
 \item
 $
 \delta^A(\xi)>n\mu(\xi)-(n-1)s(\xi),
 $
 \end{enumerate}
 then there exists a unique cscK metric in $2\pi\xi$. Here $\mu(\xi):=\frac{-K_X\cdot\xi^{n-1}}{\xi^n}$ denotes the \emph{slope} of $\xi$.
 \end{corollary}

 In some circumstances (e.g. when $c_1(X)\leq0$), the above conditions are not hard to verify; see \cite{D20} for a beautiful application. More algebraically, one can ask the following question.
 
 \begin{question}
 Let $(X,L)$ be a polarized pair. Assume that
 \begin{enumerate}
     \item 
$
 K_X+\delta(L)L\text{ is ample}
$,
\item
$
 \delta(L)>n\mu(L)-(n-1)s(L).
 $
 \end{enumerate}
 Is it true that $(X,L)$ is uniformly K-stable (in the sense of \cite{D16b})?
 \end{question}

 \subsection{The greatest Ricci lower bound for big classes}
 Let $(X,\omega)$ be a Fano manifold with a K\"ahler form $\omega\in2\pi\xi$.
 Then the greatest Ricci lower bound $\beta(\xi)$ can also be characterized by Monge--Amp\`ere equations. More precisely, we have
 \begin{equation}
     \beta(\xi)=\sup\bigg\{\beta\in(0,s(\xi)]\ \bigg|\ \omega_\varphi^n=e^{f-\beta\varphi}\omega^n\text{ is solvable for any }f\in C^\infty(X,\omega)\bigg\}.
 \end{equation}
This characterization allows us to extend the definition of $\beta(\xi)$ to big classes, as one can still make sense of Monge--Amp\`ere equations (see \cite{BEGZ}). Then it is natural to ask the following
question.
\begin{question}
Do we have
$$\beta(\xi)=\min\{s(\xi),\delta(\xi)\}$$
for any big $\RR$-line bundle $\xi$? 
\end{question}

Another interesting question is about the rationality of the greatest Ricci lower bound. It is shown in \cite{BLZ} that, $\beta(-K_X)$ is always a rational number, whose proof relies crucially on the deep analysis \cite{DS,RS} for the Gromov--Hausdorff limit arising from the continuity method. Then one can ask the following question.

\begin{question}
Is it true that
$$
\min\{s(L),\delta(L)\}\in\QQ
$$
for any big $\QQ$-line bundle $L$. 
\end{question}

Since the ample cone of a Fano manifold is a polyhedral cone (by Mori's cone theorem), $s(L)$ is easily seen to be rational. So one essentially needs to check the rationality of $\delta(L)$ (provided that $\delta(L)\leq s(L)$).

\subsection{Generalization}

Finally we remark that our approach also works for the $\theta$-twisted setting as in \cite{BBJ18}. To be more precise, let $\theta$ be a quasi-positive klt current on $X$. Then one can define $\delta_\theta$ and $\delta^A_\theta$ analogously. Indeed, $\delta_\theta$ can be defined for any ample $\RR$-line bundles using \cite[Definition 7.2]{BBJ18}. We now give the definition of $\delta^A_\theta$ for the reader's convenience. Pick a smooth representative $\theta_0\in[\theta]$, then one can write $\theta=\theta_0+\idd \psi$ for some quasi-plurisubharmonic function $\psi$ on $X$. Then for any K\"ahler class $\xi$, pick a K\"ahler form $\omega\in 2\pi\xi$ and put
$$
\delta^A_\theta(\xi):=\sup\bigg\{\lambda>0\ \bigg|\ 
    \exists\ C_\lambda>0\text{ s.t. }
    \int_Xe^{-\lambda\varphi-\psi}\omega^n\leq C_\lambda e^{\lambda J_\omega(\varphi)},\ 
    \forall\varphi\in\mathcal{H}_0(X,\omega)
    \bigg\}.
$$
The continuity of $\delta_\theta$ and $\delta^A_\theta$ can be proved following the same lines in Section \ref{sec:conti}. Note that adding a smooth form to $\theta$ will not affect $\delta_\theta$ and $\delta^A_\theta$, so one can always assume $\theta$ to be semipositive. Now to extend Theorem \ref{thm:YTD} to this setting, the only major difference in the proof is that one should work in the larger $\mathcal{E}^1$ space (we refer the reader to \cite{BBEGZ} for its definition). More precisely, in the definition of $\delta^A_\theta$ one can let $\varphi$ run through $\mathcal{E}^1(X,\omega)$, which will not change the value of $\delta^A_\theta$ by Demailly's regularization. Then one can show that (by \cite[Proposition 4.11]{BBEGZ})
$$
\delta_\theta^A(\xi)=\sup\bigg\{\lambda>0\bigg|
    \exists\ C_\lambda>0\text{ s.t. }
    \frac{1}{V}\int_X\log\frac{\omega_\varphi^n}{e^{-\psi}\omega^n}\omega_\varphi^n\geq\lambda(I_\omega-J_\omega)(\varphi)-C_\lambda,
    \forall\varphi\in\mathcal{E}^1(X,\omega)
    \bigg\}.
$$
    This in particular generalizes Proposition \ref{prop:delta-A-Mabuchi} to the $\theta$-twisted setting. And one further obtains that $\delta^A_\theta(\xi)>1$ if and only if the $\theta$-twisted functional $D_\theta/M_\theta$ is coercive on $\mathcal{E}^1(X,\omega)$, in which case one can find  $\varphi\in\mathcal{E}^1(X,\omega)$ solving $\Ric(\omega_\varphi)=\omega_\varphi+\theta$ (by \cite[Lemma 2.8]{BBJ18}). Moreover by \cite[Theorem 2.19]{BBJ18} we have a version of Proposition \ref{prop:YTD-delta-A} in this $\theta$-twisted setting as well. And also the relation $\delta^A_\theta\leq\delta_\theta$ holds as in Proposition \ref{prop:delta-A<=delta-R}.
Finally, note that the proof of Proposition \ref{prop:delta-A>1+e} can be carried out over $\mathcal{E}^1(X,\omega)$ following the lines in \cite[\S 5]{BBJ18} (see especially Lemma 5.7.(iii) in \emph{loc. cit.}). With all these ingredients combined, we obtain the following result, generalizing Theorem \ref{thm:YTD}.

\begin{theorem}
\label{thm:theta-YTD}
Let $\theta$ be a semipositive klt current $\theta$. Let $\xi$ be an ample $\RR$-line bundle such that $c_1(\xi)=c_1(X)-[\theta]$. Then we have
\begin{enumerate}
    \item If $\delta_\theta(\xi)>1$, then there exists $\omega\in2\pi c_1(\xi)$ such that $\Ric(\omega)=\omega+2\pi\theta$.
    
    \item If there exists $\omega\in2\pi c_1(\xi)$ (resp. a unique $\omega\in2\pi c_1(\xi)$) such that $\Ric(\omega)=\omega+2\pi\theta$, then $\delta_\theta(\xi)\geq1$ (resp. $\delta_\theta(\xi)>1$).
\end{enumerate}

\end{theorem}
For instance, take $\theta=[\Delta]$, where $\Delta$ is an effective $\RR$-divisor on $X$. Then $\theta$ being klt is the same as $(X,\Delta)$ being klt, in which case, the $\delta_\theta$-invariant is exactly the log $\delta$-invariant in the literature.  Now assume that $-K_X-\Delta$ is ample, then by Nadel vanishing theorem, we have $H^2(X,\mathcal{O}_X)=0$, which implies that the K\"ahler cone of $X$ coincides with the ample cone. Also note that, the solution to $\Ric(\omega)=\omega+2\pi[\Delta]$ has edge singularities along the simple normal crossing part of $\Delta$ \cite{GP}.
More generally, one can also add a smooth semi-positive form $\alpha$ to $\theta$ and consider the equation $\Ric(\omega)=\omega+\alpha+2\pi[\Delta]$. Then by Theorem \ref{thm:theta-YTD}, this equation is solvable if $\xi:=c_1(X)-[\alpha]/2\pi-[\Delta]>0$ and $\delta_\Delta(\xi)>1$. For explicit examples of such metrics, we refer the reader to \cite[Section 3]{ZZ20}.

\end{document}